\documentclass[12pt]{article}
\usepackage{amsmath, amssymb, amsthm}
\usepackage{url}

\date{March 21, 2009}
\title{Crossings, colorings, and cliques}
\author{
Michael O. Albertson \thanks{
Department of Mathematics and Statistics,
Smith College, Northampton, MA 01063. 
Email: albertson@math.smith.edu.}
\and
Daniel W. Cranston \thanks{Center for Discrete Math and Theoretical Computer Science,
Rutgers University, Piscataway, NJ 08854. E-mail:
dcransto@dimacs.rutgers.edu.}
\and
Jacob Fox \thanks{Department of Mathematics,
Princeton University, Princeton, NJ 08544. E-mail:
jacobfox@math.princeton.edu.
Research supported by an NSF Graduate Research Fellowship and a Princeton Centennial Fellowship.}
}

\newtheorem{theorem}{Theorem}
\newtheorem{lemma}{Lemma}

\newtheorem{proposition}[theorem]{Proposition}
\newtheorem*{claim*}{Claim}
\newtheorem{conj}{Conjecture}

\def\cn{\mbox{\rm cr}}

\def\floor#1{\bigg\lfloor #1 \bigg\rfloor}
\def\etal{{\em et al.}}

\begin{document}
\maketitle

\begin{abstract}
Albertson conjectured that if graph $G$ has chromatic number $r$, then the
crossing number of $G$ is at least that of the complete graph $K_r$. This conjecture in
the case $r=5$ is equivalent to the four color theorem. It was verified for $r=6$ by 
Oporowski and Zhao. 
In this paper, we prove the conjecture for $7 \leq r \leq 12$ using results of Dirac; Gallai; and Kostochka and Stiebitz that give lower bounds on the
number of edges in critical graphs, together with lower bounds by Pach \etal~on the
crossing number of graphs in terms of the number of edges and vertices. 
\end{abstract}

\section{Introduction}  

For more than a century, from Kempe through Appel and Haken and continuing to the present, the Four Color Problem~\cite{AH, RSST} has played a leading role in the development of graph theory.  
For background we recommend the book by Jensen and Toft~\cite{JT}.  
\smallskip

There are three classic relaxations of planarity.  The first is that of a graph embedded on an arbitrary surface.  Here Heawood established an upper bound for the number of colors needed to color any embedded graph.  About forty years ago Ringel and Youngs completed the work of showing that the Heawood bound is (with the exception of Klein's bottle) sharp.  Shortly thereafter Appel and Haken proved the Four Color Theorem.  One consequence of these results is that the maximum chromatic number of a graph embedded on any given surface is achieved by a complete graph.  Indeed, with the exception of the plane and Klein's bottle, a complete graph is the only critical graph with maximum chromatic number that embeds on a given surface.
\smallskip

The second classic relaxation of planarity is thickness, the minimum number of planar subgraphs needed to partition the edges of the graph.  It is well known that thickness $2$ graphs are $12$-colorable and that $K_8$ is the largest complete graph with thickness $2$.  Sulanke showed that the  $9$-chromatic join of $K_6$ and $C_5$ has thickness $2$~\cite{Ga}.  Thirty years later Boutin, Gethner, and Sulanke  \cite{BGS} constructed infinitely many $9$-chromatic critical graphs of thickness $2$.  Using Euler's Polyhedral Formula, it is straightforward to show that if $G$ has thickness $t$, then $G$ is $6t$-colorable.  When $t \ge 3$, we do not know whether complete graphs have the maximum chromatic number among all graphs of thickness $t$.  We do know that if $t \ge 3$, then $K_{6t-2}$ is the largest complete graph with thickness $t$~\cite{AG}.
\smallskip

The third classic relaxation of planarity is crossing number.  The {\it crossing number} of a graph $G$, denoted by $\cn(G)$, is defined as the minimum number of crossings in a drawing of $G$.  There are subtleties to this definition and we suggest Szekely's survey~\cite{Sz} and its references for a look at foundational issues related to the crossing number and a survey of recent results.  A bibliography of papers on crossing number can be found  at \cite{V}. Surprisingly, there are only two papers that relate crossing number with chromatic number \cite{A,OZ}.  Since these papers are not well known,  we briefly review some of their results to set the context for our work. 
\smallskip

Perhaps the first question one might ask about the connections between the chromatic number and the crossing number is whether the chromatic number is bounded by a function of the crossing number.
Albertson~\cite{A} conjectured that $\chi(G) = {\rm O(cr}(G)^{1/4})$ and this was shown by Schaefer~\cite{S}.  In Section~5, we give a short proof of this fact.
The result $\chi(G) = {\rm O(cr}(G)^{1/4})$ is best possible, since $\chi(K_n)=n$ and $\cn(K_n)\le \binom{|E(K_n)|}{2}
=\binom{\binom{n}{2}}{2} \le \frac{n^4}{8}$.
\smallskip

Although few exact values are known for the crossing number of complete graphs, the asymptotics of this problem are well-studied.  
Guy conjectured \cite{G} that the crossing number of the complete graph is as follows.
\begin{conj}[Guy]
\begin{equation}\label{complete} \cn(K_n) = \frac{1}{4}\floor{\frac{n}{2}}\floor{\frac{n-1}{2}}\floor{\frac{n-2}{2}} \floor{\frac{n-3}{2}}.
\end{equation} 
\end{conj}
He verified this conjecture for $n \leq 10$ and Pan and Richter \cite{PR} recently confirmed it for $n = 11,12$.
Let $f(n)$ denote the right hand side of equation~\eqref{complete}.  It is easy to show that $f(n)$ is an upper bound for $\cn(K_n)$, by considering a particular drawing of $K_n$ where the vertices are equally spaced around two concentric circles.

Kleitman proved that  $\lim_{n\to\infty}cr(K_n)/f(n)\geq 0.80$ \cite{kleitman}.  
Recently de Klerk \etal
~\cite{KMPRS} 
strengthened this lower bound to $0.83$.
By refining the techniques in~\cite{KMPRS}, de Klerk, Pasechnik, and Schrijver~\cite{KPS} further improved the lower bound to $0.8594$.  
\smallskip

The next natural step would be to determine exact values of the maximum chromatic number for small numbers of crossings.  An easy application of the Four Color Theorem shows that if $\cn(G) = 1$, then $\chi(G) \leq 5$. Oporowski and Zhao~\cite{OZ} showed that the conclusion also holds when $\cn(G) = 2$.  They further showed that if $\cn(G) = 3$ and $G$ does not contain a copy of $K_6$,  then $\chi(G) \leq 5$; they conjectured that this conclusion remains true even if $\cn(G)\in\{4, 5\}$.
Albertson, Heenehan, McDonough, and Wise~\cite{AHMW} showed that if $\cn(G) \leq 6$, then $\chi(G) \leq 6$.
\smallskip

The relationship between pairs of crossings was first studied by Albertson \cite{A}.  Given a drawing of graph $G$, each crossing is uniquely determined by the {\it cluster} of four vertices that are the endpoints of the crossed edges.  Two crossings are said to be {\it dependent} if the corresponding clusters have at least one vertex in common, and a set of crossings is said to be {\it independent} if no two are dependent.   Albertson gave an elementary argument proving that if $G$ is a graph that has a drawing in which all crossings are independent, then $\chi(G) \leq 6$.  He also showed that if $G$ has a drawing with three crossings that are independent, then $G$ contains an independent set of vertices one from each cluster.  Since deleting this independent set leaves a planar graph, $\chi(G) \leq 5$.   He conjectured that if $G$ has a drawing in which all crossings are independent, then $\chi(G) \leq 5$.  Independently, Wenger~\cite{W} and Harmon~\cite{H} showed that any graph with four independent crossings has an independent set of vertices one from each cluster, but there exists a graph with five independent crossings that contains no independent set of vertices one from each cluster.  Finally, Kr${\rm \acute{a}}$l and Stacho~\cite{KSta} proved the conjecture that if $G$ has a drawing in which all crossings are independent, then $\chi(G) \leq 5$.
\smallskip

At an AMS special session in Chicago in October of 2007, Albertson conjectured the following.
\begin{conj}[Albertson]
If $\chi(G) \geq r$, then $ \cn(G) \geq \cn(K_r)$.  
\label{albertson}
\end{conj}
At that meeting Schaefer observed that if $G$ contains a subdivision of $K_r$, then such a subdivision must have at least as many crossings as $K_r$ \cite{S}.  
A classic conjecture attributed to  Haj${\rm \acute{o}}$s was that if $G$ is $r$-chromatic, then $G$ contains a subdivision of $K_r$.  
Dirac \cite{D2} verified the conjecture for $r \leq 4$. 
In 1979, Catlin~\cite{C} noticed that the lexicographic product of $C_5$ and $K_3$ is an $8$-chromatic counterexample to the Haj${\rm \acute{o}}$s Conjecture.  
He generalized this construction to give counterexamples to Haj${\rm \acute{o}}$s'
 conjecture for all $r\geq 7$.
A couple of years later Erd\H{o}s and Fajtlowicz \cite{EF} proved that almost all graphs are counterexamples to Haj\'os' conjecture. However, Haj\'os' conjecture remains open for $r=5,6$. Note that if Haj${\rm \acute{o}}$s' conjecture does hold for a given $G$, then Alberton's conjecture also holds for that same $G$.  
This explains why Albertson's conjecture is sometimes referred to as the Weak Haj\'os Conjecture.
\smallskip

While exploring Conjecture~2, we've come to believe that a stronger statement is true.
Our purpose in this paper is to investigate whether for $r \geq 5$ the complete graphs are the unique critical $r$-chromatic graphs with minimum crossing number.  
While the statement of this problem is similar to that of the Heawood problem (that the chromatic number of any graph embeddable in a surface $S$ is at most the chromatic number of the largest complete graph embeddable in $S$), there are several difficulties that arise when trying to answer this problem. One difficulty is that we only conjecturally know the crossing number of the complete graph. In particular, recall that $\cn(K_n)$ is known only for $n \leq 12$ and even the results for n=11, 12 are quite recent~\cite{PR}. Furthermore, our understanding of crossing numbers for general graphs is even worse.
\smallskip

The rest of this paper is organized as follows. In Section \ref{sec2} 
we discuss known lower bounds on the number of edges in $r$-critical graphs. 
In Section \ref{sec3} we discuss known lower bounds on the crossing number, in terms of the number of edges.
In Section \ref{sec4} we prove Albertson's conjecture for $7 \leq r \leq 12$ by combining the results in the previous sections. 
In Section \ref{sec5} we show that 
any minimal counterexample to this conjecture has less than $4r$ vertices, and we also give a few concluding remarks.  

\section{Color critical graphs} 
\label{sec2}

About 1950, Dirac introduced the concept of color criticality in order to simplify graph coloring theory, and it has since led to many beautiful theorems. A graph $G$ is {\it $r$-critical} if $\chi(G)=r$ but all proper subgraphs of $G$ have chromatic number less than $r$.

Let $G$ denote an $r$-critical graph with $n$ vertices and $m$ edges.  
Define the excess $\epsilon_r(G)$ of $G$ to be 
$$\epsilon_r(G)=\sum_{x \in V(G)} \left(\deg(x)-(r-1)\right)=2m-(r-1)n.$$  Since $G$ is $r$-critical, every vertex has degree at least $r-1$ and so 
$\epsilon_r(G) \geq 0$. Brooks' theorem is equivalent to saying that equality holds if and only if $G$ is complete or an odd cycle. 
Dirac \cite{D} strengthened Brooks' theorem by proving that for $r \geq 3$, if $G$ is not complete, then $\epsilon_r(G) \geq r-3$. Later, Dirac \cite{D3} gave a complete characterization for $r \geq 4$ of those $r$-critical graphs with excess $r-3$, and, in particular, they all have $2r-1$ vertices. 
Gallai \cite{G2} proved that $r$-critical graphs that are not complete and that have at most $2r-2$ vertices have much larger excess. 
Namely, if $G$ has $n=r+p$ vertices and $2 \leq p \leq r-2$, then $\epsilon_r(G) \geq pr-p^2-2$.  A fundamental difference between Gallai's bound and Dirac's bound is that Gallai's bound grows with the number of vertices (while Dirac's does not).  
Several other papers \cite{G3,K,KS2,KS3} prove such Gallai-type bounds. 
Kostochka and Stiebitz \cite{KS} proved that if $n \geq r+2$ and $n \neq 
2r-1$, then $\epsilon_r(G) \geq 2r-6$. 
 
We will frequently use the bounds due to Dirac and to Kostochka and Stiebitz.  When we use these bounds, it will be convenient to rewrite them in terms of $m$, as below.

If $G$ is $r$-critical and not a complete graph and $r\geq 3$, then
\begin{equation*}\label{dirac}
m\geq\frac{r-1}2n+\frac{r-3}2.
\end{equation*}
We call this Dirac's bound.

If $G$ is $r$-critical, $n\geq r+2$, and $n\neq 2r-1$, then
\begin{equation*}
m\geq\frac{r-1}2n+r-3.
\end{equation*}
We call this the bound of Kostochka and Stiebitz.

We finish the section with a simple lemma classifying the $r$-critical graphs with at most $r+2$ vertices.

\begin{lemma}\label{lemma1}
For $r \geq 3$, the only $r$-critical graphs with at most $r+2$ vertices are $K_r$ and $K_{r+2} \setminus C_5$, the graph obtained from $K_{r+2}$ by deleting the edges of a cycle of length five.
\end{lemma}
\begin{proof}
The proof is by induction on $r$. For the base case $r=3$, the $3$-critical graphs are precisely odd cycles, and those with at most five vertices are 
$K_3$ and $C_5=K_5 \setminus C_5$. 

Let $G$ be an $r$-critical graph with $r \geq 4$ and $n \leq r+2$ vertices, so all vertices of $G$ have degree at least $r-1 \geq n-3$. If $G$ has a vertex $v$ adjacent to all other vertices of $G$, then clearly $G \setminus v$ is $(r-1)$-critical with at most $r+1$ vertices, and by induction, we are done in this case. So we may suppose that every vertex in the complement of $G$ has degree at least one and at most two. 
Denote by $H_1,\ldots,H_d$ the connected components of the complement of $G$. Since every vertex in the complement of $G$ has degree 1 or 2, each $H_i$ is a path or a cycle. No two vertices $u, w$ of $G$ have the same neighborhood, otherwise we could $(r-1)$-color $G \setminus u$ and give $w$ the same color as $u$. This implies that every $H_i$ that is a path has at least three edges. Every pair of vertices from different components of the complement of G are adjacent in G, and hence, have different colors in a proper coloring of G. It follows that the chromatic number of $G$ is equal to $\sum_{i=1}^d \chi_i$, where $\chi_i$ denotes the chromatic number of the subgraph of $G$ induced by the vertex set of $H_i$. Since $H_i$ is a path or a cycle, it has a matching of size at least $\lfloor |H_i|/2 \rfloor$ and hence, if $|H_i| \geq 4$,
then $\chi_i \leq \lceil |H_i|/2 \rceil \leq 3|H_i|/5$; for the final inequality here and the final inequality below, we assume that $n > 5$. Noting that if $H_i$ is a triangle then $\chi_i=1$, we have
$$\chi(G) = \sum_{i=1}^d \chi_i \leq \sum_{i=1}^d 3|H_i|/5 =3n/5 < n-2 = r,$$
contradicting that $G$ is $r$-critical and completing the proof.
\end{proof}

\section{Lower bounds on crossing number} 
\label{sec3}

A simple consequence of Euler's polyhedral formula is that every planar graph with $n \geq 3$ vertices has at most $3n-6$ edges. Suppose $G$ is a 
graph with $n$ vertices and $m$ edges. By deleting one crossing edge at a time from a drawing of $G$ until no crossing edges exist, we see that 
\begin{equation}\label{cn1} \cn(G) \geq m-(3n-6).\end{equation} 
Pach, R. Radoi\v{c}i\'c, G. Tardos, and G. T\'oth \cite{PRTT} proved the following lower bounds on the crossing number. 
\begin{eqnarray}\label{cn2} \cn(G) & \geq & \frac{7}{3}m-\frac{25}{3}(n-2), \\ 
\label{cn3} \cn(G) & \geq & 3m-\frac{35}{3}(n-2), \\ 
\label{cn4} \cn(G) & \geq & 4m-\frac{103}{6}(n-2).\end{eqnarray} 

Although inequality (\ref{cn3}) is not written explicitly in \cite{PRTT}, 
it follows from their proof of (\ref{cn4}). 
Of the above four inequalities on the crossing number, inequality (\ref{cn1}) is best when $m \leq 4(n-2)$, inequality (\ref{cn2}) is best when $4(n-2) \leq m \leq 5(n-2)$, inequality (\ref{cn3}) is best 
for $5(n-2) \leq m \leq 5.5(n-2)$, and inequality (\ref{cn4}) is best when $m \geq 5.5(n-2)$. 

A celebrated result of Ajtai, Chv\'atal, Newborn, and Szemer\'edi \cite{ACNS} 
and Leighton~\cite{L}, known as {\em the Crossing Lemma}, states
that the crossing number of every graph $G$ with $n$ vertices and $m
\geq 4n$ edges satisfies
\begin{equation*} 
\cn(G) \geq \frac{1}{64}\frac{m^3}{n^2}.
\end{equation*}
The constant factor $\frac{1}{64}$ comes from the well-known probabilistic proof~\cite{AZ} using inequality (\ref{cn1}).  
The best known constant factor is due to Pach \etal~\cite{PRTT}. Using (\ref{cn4}), they show for $m \geq \frac{103}{16}n$ that
\begin{equation}\label{crosslemma2}\cn(G) \geq \frac{1}{31.1}\frac{m^3}{n^2}.\end{equation}

\section{Albertson's conjecture for $r \leq 12$} 
\label{sec4}

In this section we prove Albertson's conjecture (Conjecture~\ref{albertson}) 
for $r=7,8,9,10,11,12$.
Note that if $H$ is a subgraph of $G$, then $\cn(H) \leq \cn(G)$. Therefore, to prove Albertson's conjecture for a given $r$, it suffices to prove it 
only for $r$-critical graphs. 

Lemma \ref{lemma1} demonstrates that the only $r$-critical graphs with $n \leq r+2$ vertices are $K_r$ and $K_{r+2} \setminus C_5$. This second graph contains a subdivision of $K_r$. Indeed, by using all the vertices of $K_{r+2}\setminus C_5$ and picking two adjacent vertices of degree $r-1$ to be internal vertices of a subdivided edge, we get a subdivision of $K_r$ with only one subdivided edge. Hence, $\cn(K_{r+2} \setminus C_5) \geq \cn(K_r)$. So a counterexample to Albertson's conjecture must have at least $r+3$ vertices.  However, none of our proofs rely on this observation except for the proof of Proposition~\ref{prop12case}; the others use only the easier observation that no $r$-critical graph has $r+1$ vertices.

\begin{proposition}\label{prop7case}
If $\chi(G)=7$, then $\cn(G) \geq 9 = \cn(K_7)$.  
\end{proposition}
\begin{proof}
By the remarks above, we may suppose $G$ is $7$-critical and not $K_7$. Let $n$ be the number of vertices of $G$ and $m$ be the number of edges of $G$.
By Dirac's bound, we have $m \geq 3n+2$. Borodin \cite{Bo} showed that if a graph has a drawing in the plane in which each edge intersects at most one other edge, then the graph has chromatic number at most $6$. Consider a drawing $D$ of $G$ in the plane with $\cn(G)$ crossings. Since $G$ has chromatic number $7$, there is an edge $e$ in $D$ that intersects at least two other edges.
Beginning with $e$, we delete one crossing edge at a time, until no crossing edges exist. We get that $\cn(G)\geq m - (3n-6) + 1= m-3n+7$.  Since $m\geq 3n+2$, this bound gives:
$$\cn(G) \geq m-3n+7 \geq 9.$$
This completes the proof.  
\end{proof}

\begin{proposition}\label{prop8case}
If $\chi(G)=8$ and $G$ does not contain $K_8$, then $\cn(G) \geq 20 > 18 = \cn(K_8)$.
\end{proposition}
\begin{proof}
We may suppose $G$ is $8$-critical. Let $n$ be the number of vertices of $G$ and $m$ be the number of edges of $G$.
When $n = 15$, Dirac's bound 
gives $m \geq \frac{7}{2}n+2.5 = 55$, and thus inequality (\ref{cn2}) gives 
$$\cn(G) \geq \frac{7}{3}m-\frac{25}{3}(n-2)\ge 20.$$
When $n \not = 15$, the bound of Kostochka and Stiebitz 
gives $m \geq \frac{7}{2}n+5$.
When we substitute for $m$, inequalities (\ref{cn2}) and (\ref{cn3}) give
$$\cn(G) \geq m-3n+6 \geq \frac{n}{2}+11,$$
and
$$\cn(G) \geq \frac{7}{3}m-\frac{25}{3}(n-2) \geq \frac{7}{3}(\frac{7}{2}n+5)-\frac{25}{3}(n-2)=-\frac{n}{6}+\frac{85}{3}.$$
The first lower bound shows that $\cn(G) \geq 20$ if $n \geq 18$, while the second lower bound shows that $\cn(G) \geq 20$ if $n \leq 50$.  
This completes the proof.
\end{proof}

\begin{proposition}\label{prop9case}
If $\chi(G)=9$ and $G$ does not contain $K_9$, then $\cn(G) \geq 41 > 36 = \cn(K_9)$.
\end{proposition}
\begin{proof}
We may suppose $G$ is $9$-critical. Let $n \geq 11$ be the number of vertices of $G$ and $m$ be the number of edges of $G$.
When $n = 17$, Dirac's bound 
gives $m \geq 4n+3 = 71$, so inequality (\ref{cn2}) gives
$$\cn(G) \geq \frac{7}{3}m-\frac{25}{3}(n-2)\ge \frac{122}{3} > 40.$$
Thus $\cn(G) \geq 41$. 
When $n \not = 17$, the bound of Kostochka and Stiebitz 
gives $m \geq 4n+6$.
Hence, inequality (\ref{cn2}) gives 
$$\cn(G) \geq \frac{7}{3}m-\frac{25}{3}(n-2) \geq n +\frac{92}{3} \geq 11+\frac{92}{3} > 41.$$ 
This completes the proof. 
\end{proof}

\begin{proposition}\label{prop10case}
If $\chi(G)=10$ and $G$ does not contain $K_{10}$, then $\cn(G) \geq 69 > 60 = \cn(K_{10})$.
\end{proposition}
\begin{proof}
We may suppose $G$ is $10$-critical. Let $n \geq 12$ be the number of vertices of $G$ and $m$ be the number of edges of $G$.
When $n = 19$, Dirac's bound 
gives $m \geq \frac{9}{2}n+\frac{7}{2} = 89$, so inequality (\ref{cn3}) gives
$$\cn(G) \geq 3m-\frac{35}{3}(n-2)\ge\frac{206}{3} > 68.$$
Thus $\cn(G) \geq 69$.
When $n \not = 19$, the bound of Kostochka and Stiebitz 
gives $m \geq \frac{9}{2}n+7$, so
inequality (\ref{cn4}) gives 
$$\cn(G) \geq 4m-\frac{103}{6}(n-2) \geq \frac{5}{6}n +\frac{187}{3} \geq 10+\frac{187}{3} > 72.$$
This completes the proof.
\end{proof}

\begin{proposition}\label{prop11case}
If $\chi(G)=11$ and $G$ does not contain $K_{11}$, then $\cn(G) \geq 104 > 100 = \cn(K_{11})$.
\end{proposition}
\begin{proof}
We may suppose $G$ is $11$-critical. Let $n \geq 13$ be the number of vertices of $G$ and $m$ be the number of edges of $G$.
When $n = 21$, Dirac's bound 
gives $m \geq 5n+4 = 109$, so inequality (\ref{cn4}) gives
$$\cn(G) \geq 4m-\frac{103}{6}(n-2)\ge\frac{659}{6} > 109.$$
Thus $\cn(G) \geq 110$.
When $n \not = 21$, the bound of Kostochka and Stiebitz 
gives $m \geq 5n+8$, so
inequality (\ref{cn4}) gives
$$\cn(G) \geq 4m-\frac{103}{6}(n-2) \geq \frac{17}{6}n +\frac{199}{3} \geq \frac{17}{6} \cdot 13+\frac{199}{3} > 103.$$ 
Thus $\cn(G) \geq 104$, 
which completes the proof.
\end{proof}

\begin{proposition}\label{prop12case}
If $\chi(G)=12$, then $\cn(G) \geq 153 > 150 = \cn(K_{12})$.
\end{proposition}
\begin{proof}
We may suppose $G$ is $12$-critical and is not $K_{12}$. Let $n$ be the number of vertices of $G$ and $m$ be the number of edges of $G$. By the remark before the proof of Proposition \ref{prop7case}, we may suppose $G$ has at least 15 vertices.

{\bf Case 1:} $n = 23$. Dirac's bound gives $m \geq \frac{11}{2}n+\frac{9}{2} = 131$, so inequality (\ref{cn4}) gives
$$\cn(G) \geq 4m-\frac{103}{6}(n-2)\ge\frac{327}{2} > 163.$$
Thus $\cn(G) \geq 164$.

{\bf Case 2:} $n>16$ and $n \not = 23$. The bound of Kostochka and Stiebitz 
gives $m \geq \frac{11}{2}n+9$, so inequality (\ref{cn4}) gives
$$\cn(G) \geq 4m-\frac{103}{6}(n-2) \geq \frac{29}{6}n +\frac{211}{3} > 152.$$
Thus we get $\cn(G) \geq 153$ if $n > 16$.
 
{\bf Case 3:} $n=15$. By rewriting Gallai's bound (from Section~\ref{sec2}) as a lower bound on $m$, and substituting $r=12$,
we get the inequality $m \geq \frac{11}{2}n+\frac{3}{2}r-\frac{11}{2} = \frac{11}{2}n+\frac{25}{2}=95$.
Now inequality (\ref{cn4}) gives 
$$\cn(G) \geq 4m-\frac{103}{6}(n-2) \ge 4 \cdot 95 - \frac{103}{6} \cdot 13 > 156.$$

{\bf Case 4:} $n=16$. We again use Gallai's bound 
with $r=12$, 
and now we get the inequality $m \geq \frac{11}{2}n+2r-9 = 103$.
Now inequality (\ref{cn4}) gives 
$$\cn(G) \geq 4m-\frac{103}{6}(n-2) \ge 4\cdot 103 - \frac{103}6\cdot 14 > 171.$$
This completes the proof. 
\end{proof}

\section{Concluding remarks} 
\label{sec5}

In the previous section, we showed that a minimal counterexample to Albertson's conjecture has at least $r+3$ vertices. Here we give an upper bound 
on the number of vertices of a counterexample. 

\begin{proposition}
If $G$ is an $r$-critical graph with $n \geq 4r$ vertices, 
then $\cn(G) \geq \cn(K_r)$.
\end{proposition}
\begin{proof}
We have shown that $\cn(G)\ge \cn(K_r)$ holds for $r\le 12$.  If $r=13$, then inequality~\eqref{cn4} easily implies the proposition.  Thus, we may assume $r\ge 14$.  
Let $m$ be the number of edges in $G$.
Since $G$ is $r$-critical, $m \geq n(r-1)/2$. In particular, $m \geq 6.5n > \frac{103}{16}n$. Therefore, the bound (\ref{crosslemma2}) gives 
\begin{align*}
\cn(G) &\geq \frac{1}{31.1}\frac{m^3}{n^2} \geq \frac{1}{8 \cdot 31.1}(r-1)^3n \geq \frac{1}{64}(r-1)^3r \\
& ~\vspace{-1in}\\
&\geq 
\frac{1}{4}\floor{\frac{r}{2}}\floor{\frac{r-1}{2}} \floor{\frac{r-2}{2}} \floor{\frac{r-3}{2}}\geq \cn(K_r).
\end{align*}
\end{proof}

Without assuming any lower bound on $n$,
we can prove that $\cn(G) \geq (r-1)^4/2^8$ if $G$ has chromatic number $r \geq 14$.  This immediately implies $\chi(G)  \leq 1+4\cn(G)^{1/4}$. 

We think that if $G$ has chromatic number $r$ and does not contain $K_r$, then $\cn(G)-\cn(K_r)$ is not only nonnegative, but is at least cubic in $r$. Recall that $K_{r+2} \setminus C_5$ is $r$-critical and note that it is a subgraph of $K_{r+2}$; hence, if Guy's conjecture on the crossing number of $K_r$ is true, then $K_{r+2}\setminus C_5$ shows that $\cn(G)-\cn(K_r)$ can be as small as cubic in $r$. 

\paragraph{\bf Acknowledgment.}
We thank Sasha Kostochka for helpful discussions on excess in critical graphs.


\begin{thebibliography}{99}
\bibitem{AZ}
M. Aigner and G. Ziegler,
Proofs from the book, Springer-Verlag, New York, 2004.

\bibitem{ACNS} 
M. Ajtai, V. Chv\'atal, M. Newborn, and E. Szemer\'edi, 
{Crossing-free subgraphs, in} \emph{ Theory and Practice of
Combinatorics,} vol.~60 of Math. Studies, North-Holland,
Amsterdam, 1982, pp.~9--12.

\bibitem{A} 
M. O. Albertson, Chromatic number, independence ratio, and crossing number, \emph{ Ars Mathematica. Contemporanea,} \textbf{ 1} (2008), 1--6. 

\bibitem{AHMW} 
M. O. Albertson, M. Heenehan, A. McDonough, and J. Wise, 
Coloring graphs with given crossing patterns, (manuscript).

\bibitem{AG} V.B. Alekseev and V.S. Go\v{n}cakov, The thickness of an arbitrary complete graph. (Russian)  \emph{Mat. Sb. (N.S.)} {\bf 101} (143) (1976), no. 2, 212--230.  English translation: \emph{Math. USSR-Sb.} {\bf 30} (1976), no. 2, 187--202 (1978).  

\bibitem{AH} 
K. Appel and W. Haken, Every planar map is four colorable, Part I. Discharging, \emph{ Illinois J. Math.} \textbf{21} (1977), 429--490. 

\bibitem{Bo} 
O. V. Borodin, Solution of the Ringel problem on vertex-face coloring of planar graphs and coloring of $1$-planar graphs. (Russian) \emph{Metody Diskret. 
Analiz.} No. 41 (1984), 12--26.

\bibitem{BGS} 
D. L. Boutin, E. Gethner, and T. Sulanke,  
Thickness-two graphs. I. New nine-critical graphs, permuted layer graphs, and Catlin's graphs. \emph{ J. Graph Theory} \textbf{57} (2008), no. 3, 198--214. 

\bibitem{B} 
R. L. Brooks, On colouring the nodes of a network, \emph{ Proc. Camb. Phil. Soc.} \textbf{37} (1941), 194--197. 

\bibitem{C} 
P. A. Catlin, Haj\'os' graph-coloring conjecture: variations and counterexamples, \emph{ J. Combin. Theory Ser. B} \textbf{26} (1979), 268--274.

\bibitem{D2}
G. A. Dirac, A property of 4-chromatic graphs and some remarks on critical graphs, \emph{ J. London Math. Soc.} \textbf{27} (1952), 85--92.

\bibitem{D} 
G. A. Dirac, 
A theorem of R. L. Brooks and a conjecture of H. Hadwiger,
\emph{ Proc. London Math. Soc.} (3) \textbf{7} (1957), 161--195.

\bibitem{D3} 
G. A. Dirac, The number of edges in critical graphs.
\emph{ J. Reine Angew. Math.} \textbf{268/269} (1974), 150--164. 

\bibitem{EF} 
P. Erd\H{o}s and S. Fajtlowicz, On the conjecture of Haj\'os, \emph{Combinatorica} \textbf{1} (1981), 141--143.

\bibitem{G3}
T. Gallai, Kritische Graphen. I. (German) \emph{Magyar Tud. Akad. Mat. Kutató Int. K\"ozl.} \textbf{8} (1963), 165--192.

\bibitem{G2} 
T. Gallai, Kritische Graphen. II. (German) \emph{Magyar Tud. Akad. Mat. Kutató Int. K\"ozl.} \textbf{8} (1963), 373--395.

\bibitem{Ga}
M. Gardner, Mathematical Games, \emph{Scientific American.} \textbf{242} (Feb. 1980), 14--19.

\bibitem{G} 
R. K. Guy, Crossing Numbers of Graphs. In \emph{Graph Theory and Applications: Proceedings of the Conference at Western Michigan University, Kalamazoo, Mich., May 10-13, 1972 (Ed. Y. Alavi, D. R. Lick, and A. T. White).} New York: Springer-Verlag, pp. 111--124, 1972.

\bibitem{H} 
N. Harmon, 
Graphs with four independent crossings are five colorable, 
\emph{Rose-Hulman Undergraduate Mathematics Journal} \textbf{9} (2008), available from: \url{http://www.rose-hulman.edu/mathjournal/archives/2008/vol9-n2/paper12/v9n2-12p.pdf}, retrieved on 15 January 2009.

\bibitem{JT} 
T. R. Jensen and B. Toft, 
Graph coloring problems, John Wiley \& Sons, New York, 1995.

\bibitem{kleitman}
D. Kleitman, 
The crossing number of $K_{5,n}$, \emph{J. Combin. Theory} \textbf{9} (1970), 315--323.

\bibitem{KMPRS} 
E. de Klerk, J. Maharry, D.V. Pasechnik, R. B. Richter, and G. Salazar, 
Improved bounds for the crossing numbers of $K_{m,n}$ and $K_n$. \emph{SIAM J. Discrete Math.} \textbf{20}, (2006), 189--202.

\bibitem{KPS} 
E. de Klerk, D.V. Pasechnik, and A. Schrijver, 
Reduction of symmetric semidefinite programs using the 
$*$-representation, \emph{Math Program Ser. B} \textbf{109} (2007), 613--624.

\bibitem{KS3}
A. V. Kostochka and M. Stiebitz,
Colour-critical graphs with few edges, \emph{Discrete Math.} \textbf{191} (1998), 125--137. 

\bibitem{KS} 
A. V. Kostochka and M. Stiebitz,
Excess in colour-critical graphs, Graph theory and combinatorial biology (Balatonlelle, 1996), 87--99,
Bolyai Soc. Math. Stud., 7, J\'{a}nos Bolyai Math. Soc., Budapest, 1999.

\bibitem{KS2}
A. V. Kostochka and M. Stiebitz,
A new lower bound on the number of edges in colour-critical graphs and hypergraphs. \emph{J. Combin. Theory Ser. B} \textbf{87} (2003), 374--402.

\bibitem{KSta} 
D. Kr${\rm \acute{a}}$l and L. Stacho,  
Coloring plane graphs with independent crossings, preprint, available from: \url{http://kam.mff.cuni.cz/~kamserie/serie/clanky/2008/s886.ps}, retrieved on 15 January 2009.

\bibitem{K} 
M. Krivelevich, On the minimal number of edges in color-critical graphs, \emph{Combinatorica} \textbf{17} (1997), 401--426.

\bibitem{L}
F. T. Leighton, {New lower bound techniques for VLSI,} \emph{Math.
Systems Theory} \textbf{17} (1984), 47--70.

\bibitem{OZ} 
B. Oporowski and D. Zhao, 
Coloring graphs with crossings, arXiv:math/0501427 [math.CO] 25 Jan 2005. 

\bibitem{PRTT} 
J. Pach, R. Radoi\v{c}i\'c, G. Tardos, G. T\'oth, Improving the crossing lemma by finding more crossings in sparse graphs, \emph{Discrete and Computational Geometry} \textbf{36} (2006), 527--552.

\bibitem{PR} 
S. Pan and R. B. Richter, 
The crossing number of $K_{11}$ is $100$, \emph{J. Graph Theory} \textbf{56} (2007), 128--134. 

\bibitem{RSST} 
N. Robertson, D. P. Sanders, P. D. Seymour and R. Thomas, The four-color theorem, \emph{J. Combin. Theory Ser. B} \textbf{70} (1997), 2--44.

\bibitem{S} 
M. Schaefer, 
personal communication.

\bibitem{Sz}  
L. Sz\'{e}kely, 
A successful concept for measuring non-planarity of graphs: the crossing number,  \emph{Discrete Math.} \textbf{276} (2004), 331--352. 

\bibitem{V} 
I. Vrt'o, 
Crossing Numbers of Graphs: A Bibliography, available from: \url{ftp://ftp.ifi.savba.sk/pub/imrich/crobib.pdf}, retrieved on 15 January 2009.
 
\bibitem{W} 
P. Wenger, 
Independent crossings and chromatic number, (manuscript),
available from: \url{http://www.math.uiuc.edu/~pwenger2/independentcrossing.pdf}, retrieved on 26 January 2009.

\end{thebibliography}
\end{document}